\documentclass[12pt]{amsart}
\usepackage{color,enumerate, mathrsfs,upref}
\usepackage[colorinlistoftodos]{todonotes}

\newtheorem{theorem}{Theorem}[section]  
\newtheorem{lemma}[theorem]{Lemma}
\newtheorem{proposition}[theorem]{Proposition}

\newtheorem*{theorem*}{Theorem}
\theoremstyle{definition}
\newtheorem{definition}[theorem]{Definition}
\newtheorem{example}[theorem]{Example}

\theoremstyle{remark}
\newtheorem{remark}[theorem]{Remark}
\numberwithin{equation}{section}

	\usepackage{amssymb}
	\graphicspath{ {./images/} }

	\newcommand{\tand}{\text{ and }}
	\newcommand{\tif}{\text{ if }}
	\newcommand{\eqtext}[1]{\ensuremath{\stackrel{\text{#1}}{=}}}
	\newcommand{\ssp}[2]{\hspace{#1mm} #2 \hspace{#1mm}} 

    \newcommand{\KP}{\texttt{KP}}
    \def\St{\mathcal{S}}
    \newcommand{\Sh}{\widehat{\mathcal{S}}}
    \def\le{\leqslant}
    \def\ge{\geqslant}
	
	\def\Im{\text{Im }}
	\def\Arg{\text{Arg }}
	\def\diam{\text{diam }}
	\newcommand{\dist}{\text{dist}}
    \newcommand{\D}{\mathbb{D}}
    \newcommand{\onto}{\xrightarrow{\textnormal{onto}}}

\title[Disk to Domain Lipschitz Estimates]{Lipschitz Estimates for Conformal Maps from The~Unit Disk to Convex Domains}

\author{Christopher G. Donohue} 
\address{215 Carnegie, Department of Mathematics, Syracuse University, Syracuse, NY 13244, USA} 
\email{cgdonohu@syr.edu}


\begin{document}
\baselineskip6.5mm

\begin{abstract}
We obtain an explicit uniform upper bound for the derivative of a conformal mapping of the unit disk onto a convex domain. This estimate depends only on the outer and inner radii of the domain, and on the minimum curvature radius of its boundary. Its proof is based on a M\"obius invariant metric of hyperbolic type, introduced by Kulkarni and Pinkall in 1994.
\end{abstract}

\maketitle

\section{Introduction}\label{sec:intro} 

Estimates for the derivatives of conformal maps, collectively known as distortion theorems, have historically been given on compact subsets of the domain~\cite[\S2.3]{Duren-book}. Uniform bounds for the derivative have received less attention. The classical theorem of Kellogg~\cite[Theorem~3.5]{Pommerenke-book} states that the derivative of a conformal map $f$ between $C^{1,\alpha}$ domains is uniformly bounded, but does not provide an explicit upper bound. 
Such an upper bound was recently obtained in~\cite{Kovalev2017} and was found to be useful both in the spectral theory of PDEs with applications to quantum  physics~\cite{Lotoreichik}, and as a global upper bound on the integral Hardy norm of $f'$ in fluid dynamics~\cite{Hassainia}. 

This paper improves the uniform upper bound on $|f'|$ from~\cite{Kovalev2017}. Our main result is stated below. 

\begin{theorem*}
Let $\Omega$ satisfy the $(R_O,R_I,R_C)$ condition in Definition~\ref{RadiiDefinition}. Then for any conformal map $f\colon  \D\onto\Omega$ fixing $0$ we have
\[
\|f'\|_{H^\infty} \le R_Ce^{2F(R_O,R_I,R_C)}
\]
where $F(R_O,R_I,R_C)$ is as in Theorem~\ref{DistanceBounds}. Equality is attained whenever $\Omega$ is a disk. 
\end{theorem*}

For comparison, the bound in~\cite{Kovalev2017} is sharp only for disks centered $0$, that is, only when $f$ is a linear function. The main tool we use to estimate $f'$ is the Kulkarni-Pinkall metric~\cite{KulkarniPinkall} which is defined in Section~\ref{sec:prelim}. A precise estimate for this metric is derived in Section~\ref{sec:hyperbolic}. Section~\ref{sec:derivative} contains the proof of the main result. The paper concludes with examples in Section~\ref{sec:examples}. 

\section{Definitions and preliminary results}\label{sec:prelim}

Throughout this section let $\Omega\subsetneq \mathbb C$ be a simply connected domain. 
The \emph{hyperbolic metric}~\cite[\S4.6]{Pommerenke-book} of $\Omega$ is conformally invariant and has constant Gaussian curvature $-4$. The hyperbolic distance between $z,w\in \Omega$ is denoted $h_{\Omega}(z,w)$, and the density at $z$ by $\lambda_{\Omega}(z)$. When $\Omega$ is a disk of radius $r$ and $z$ is a point at distance $d$ from its center, we have 
\begin{equation}\label{hyp-density-in-disk}
 \lambda_{\Omega}(z) = \frac{r}{r^2-d^2}.   
\end{equation}
For more general $\Omega$ however, explicit formulas for $\lambda_{\Omega}$ or $h_{\Omega}$ are tied to explicit conformal maps between $\Omega$ and the unit disk---in most cases neither exist. For this reason alternative metrics are used as approximations to the hyperbolic~\cite{HKMbook}. The most common is the \emph{quasihyperbolic metric} with density $1/\dist(z, \partial\Omega)$. Quasihyperbolic
distance,  usually denoted $k(z,w)$, employs the surprisingly effective simplification of substituting inverse distance to the boundary for the hyperbolic density. Using our $-4$ curvature convention, we have the well known comparison $\frac14 k\le h \le k$. The quasihyperbolic metric was used to attain the bound in~\cite{Kovalev2017}. We will improve this bound by using a more refined metric in its place.

\begin{definition}\label{def:KPmetric}
Distance in the \emph{Kulkarni-Pinkall (KP) metric} between $z,w\in \Omega$ is denoted by $\KP_{\Omega}(z,w)$, the density by $\mu_{\Omega}(z)$. If we take $\Delta$ to be the set of all disks $D$ such that $z\in D\subset \Omega$, then
\begin{equation}\label{inf-def-mu}
\mu_{\Omega}(z)\eqtext{def}\inf_{D\in\Delta}\lambda_D(z).
\end{equation}
\end{definition}

The KP density is  comparable to the hyperbolic density~\cite[\S3]{HMM2005},
\begin{equation}\label{KPcomparison}
\frac12 \mu \le \lambda \le \mu,
\end{equation}
again shown with the $-4$ curvature convention. Note that the KP metric gives a better approximation to the hyperbolic metric than the quasihyperbolic does.

The KP metric was introduced by Kulkarni and Pinkall in a 1994 article~\cite{KulkarniPinkall} with an emphasis on its M\"obius invariance.  In different ways the KP and quasihyperbolic metrics both take advantage of the fact that the hyperbolic metric is monotone with respect to domain, that is, if $\Omega_1\subset \Omega_2 \subset\mathbb{C}$ are simply connected domains then
\begin{equation}\label{eq:HyperbolicDomainMonotonic}
\forall a,b\in \Omega_1,\quad h_{\Omega_1}(a,b)\ge h_{\Omega_2}(a,b).
\end{equation}
This can be seen as a consequence of the Schwarz-Pick lemma.

By~\cite[Theorem 3.5]{HMM2003} for each point $z$ in a simply connected domain $\Omega\subsetneq\mathbb C$, there exists a unique disk that attains the infimum in~\eqref{inf-def-mu}, referred to as the \textit{extremal disk} (this disk is understood in the sense of the Riemann sphere when $\Omega$ is unbounded). The extremal disk is determined by a subtle trade-off between the size of the disk and the proximity of its center to $z$.

\begin{lemma}\label{lemma-extremal-disk}~\cite[\S2.3]{HMM2005}
For a simply connected domain $\Omega\subsetneq\mathbb C$ and a point $z\in \Omega$, the $\KP_\Omega$~extremal disk for $z$ is the unique disk $D$ satisfying the condition that $z$ lies in the closure of the convex hull of $\partial D\cap \partial \Omega$ with respect to the hyperbolic metric on $D$.
\end{lemma}

\begin{remark}\label{rem:ExtremalDiskDomainMonotone}
Suppose $\Omega_1\subset\Omega_2$ are domains and $D$ is the $\KP_{\Omega_2}$~extremal disk for some $z\in\Omega_2$. If $D\subset\Omega_1$, then $D$ is also the $\KP_{\Omega_1}$~extremal disk for $z$.
\end{remark}

\begin{example}\label{ex:ExtremalDiskInStrip}
The $\KP_H$~extremal disk at $x\in\mathbb{R}$ for the infinite strip $H \eqtext{def} \{|\Im z| <1\}$ is $D \eqtext{def} \{z\colon|x-z|<1\}$ with $\partial D\cap\partial H = \{x\pm i\}$. Here $(x-i,x+i)$ is a hyperbolic geodesic in $D$ and is the hyperbolic convex hull of $\partial D\cap\partial H$ in $D$.
\end{example}

\begin{example}\label{ex:ExtremalDiskInSector}
Fix $\theta\in(0,\pi/2)$. The $\KP_S$~extremal disk at $x>0$ for the sector $S\eqtext{def}\{|\Arg z|<\theta\}$ is the disk $D$ with $\partial D\cap\partial S=\{xe^{\pm i\theta}\}$ and $A\eqtext{def}\{xe^{it}:|t|<\theta\}\subset D$. Here $A$ is a hyperbolic geodesic in $D$ and is the hyperbolic convex hull of $\partial D\cap\partial S$ in~$D$.
\end{example}

\begin{example}\label{ex:ExtremalDiskInDisk}
The $\KP$~extremal disk for every point in a domain $D$ that is itself a disk is $D$, because the convex hull of $\partial D$ with respect to the hyperbolic metric on $D$ is all of $D$.
\end{example}

\begin{example}\label{ex:HullOfCircularArc}
Suppose a domain $S$ contains a disk $D$ such that $\Gamma\eqtext{def}\partial D\cap\partial S$ is a circular arc. Then the convex hull of $\Gamma$ in the hyperbolic metric on $D$ is the portion of $D$ bounded by $\Gamma$ and a circle orthogonal to $\Gamma$ at both of its endpoints. Here we use circle in the sense of the Riemann sphere so that if $\Gamma$ is a semicircle, the orthogonal circle is a line.
\end{example}

The inclusion of Remark~\ref{rem:ExtremalDiskDomainMonotone} and subsequent examples are to clarify the extremal disks on the segment between centers of a stadium as described in Definition~\ref{definition-stadium}, and in the proof of Lemma~\ref{KPdistance}. Throughout the paper $\D$ is the unit disk and $D(a,r)\eqtext{def} \{z\colon |a-z|<r\}$.

\begin{definition}\label{RadiiDefinition}
Suppose $\Omega$ is a simply connected \emph{convex} domain that contains~$0$ and has $C^{1,1}$-smooth boundary.  We say that such a domain satisfies the $(R_O, R_I, R_C)$ condition if: 
\begin{itemize}
    \item $R_O$, $R_I$, $R_C$ are all positive,
    \item $R_O$ is the minimal $r$ such that $\Omega\subset D(0,r)$, 
    \item $R_I$ is the maximal $r$ such that $D(0,r)\subset\Omega$,
    \item $\Omega$ can be expressed as a union of open disks of radius~$R_C$.
\end{itemize}
\end{definition}

The subscripts in Definition~\ref{RadiiDefinition} serve to indicate that $R_O$ is the \emph{outer radius}, $R_I$ the \emph{inner radius}, and $R_C$ a \emph{curvature radius}. The following remark clarifies the geometric meaning of $R_C$. 

\begin{remark}\label{rem:CurvatureRadius}  Suppose $\Omega$ and $R_C$ are as in  Definition~\ref{RadiiDefinition}. Then for every $w\in \partial \Omega$ there exists $a\in\Omega$ such that $D(a, R_C)\subset \Omega$ and $w\in \partial D(a, R_C)$. 

Indeed, one can take a sequence $z_n\to w$ of points $z_n\in \Omega$ and cover each $z_n$ with a disk $D(a_n, R_C)\subset \Omega $. The sequence $\{a_n\}$ has a convergent subsequence $\{a_{n_k}\}$. Let $a$ be its limit. Clearly $D(a, R_C)\subset  \Omega$, which implies $|a-w|\ge R_C$. On the other hand, $|a-w|=\lim_{k\to\infty} |a_{n_k}-z_{n_k}|\le R_C$. 
\end{remark}

We denote and define the Hardy norm of a holomorphic function $f$ on $\D$ by 
\[
\|f\|_{H^p} \eqtext{def} \sup_{0<r<1}\left(\int_0^{2\pi} |f(re^{i\theta})|^p \frac{d\theta}{2\pi}  \right)^{1/p}, \quad  \|f\|_{H^\infty} = \sup_{\D}|f|.
\]
Using the KP metric we will improve the derivative bound from~\cite{Kovalev2017}, which relied on the quasihyperbolic metric and can be stated as 
\begin{equation}\label{old-bound}
\|f'\|_{H^\infty} \le R_C\exp\{2(R_O-R_C)\Phi(R_I,R_C)\}
\end{equation}
where \[
\Phi(a,b)\eqtext{def}\left\{\begin{array}{cc} \frac{\log a - \log b}{a-b}, &\tif a\ne b\\\frac{1}{a}, &\tif a=b.\end{array}\right.
\]
according to~\cite[Proposition~19]{Lotoreichik}. Our improved bound (Theorem~\ref{Main}) is sharp in a wider class of convex domains than~\eqref{old-bound}.

\section{Estimates for the hyperbolic metric in convex domains}\label{sec:hyperbolic} 

We introduce a class of convex domains which are convenient for estimating the hyperbolic metric.

\begin{definition}\label{definition-stadium} A \emph{stadium} is the convex hull of the union of two open disks in the plane. It is denoted $\St(r_1,r_2,d)$ where $r_1$ and $r_2$ are the radii of the two disks and $d$ is the distance between their centers. 
\end{definition}

The notation $\St(r_1,r_2,d)$ in Definition~\ref{definition-stadium} omits the centers of the disks that form the stadium since they are usually irrelevant to the hyperbolic geometry of the domain. The centers will be given in context when relevant.

\begin{lemma}\label{lemma-stadium boundary smooth} The boundary of a stadium is $C^{1,1}$-smooth. That is, the unit speed parameterization of its boundary has Lipschitz continuous derivative. 
\end{lemma} 

\begin{proof}
The boundary of a stadium $\St(r_1,r_2,d)$ consists of circular arcs, possibly joined by tangent line segments. If we take $w$ to be a unit speed parameterization of the boundary and $r=\min(r_1, r_2)$, then the inequality
\begin{equation}\label{eq-lip-deriv}
|w'(t_1)-w'(t_2)|\le\frac{1}{r}|t_1-t_2|    
\end{equation}
holds on each of two circular arcs. It also holds on linear segments where $w'$ is constant. It now follows that $w'$ is Lipschitz continuous.
\end{proof}

\begin{definition}\label{def:infinite-sector}
Let $D_{r_1}$ and $D_{r_2}$ denote the two open disks from Definition~\ref{definition-stadium}. If $d>|r_2-r_1|$, then the boundary of the stadium $\St\eqtext{def} \St(r_1,r_2,d)$ is composed of two circular arcs and two congruent line segments. These segments can be extended to circumscribe an \emph{infinite sector} $\Sh$ around $\St$. The opening of the sector is $2\theta$ where $\theta\eqtext{def}\arcsin\frac{|r_2-r_1|}{d}$. When working with a sector $\Sh(r_1,r_2,d)$ it is useful to assume that after a rigid motion, $\Sh=\{z\colon |\Arg z|<\theta\}$.
\end{definition}

\begin{lemma}\label{KPdistance}
Given a stadium $\St(R,r,d)$ where $r\le R$, let $\theta=\arcsin\frac{R-r}{d}$ if $d\ge R-r$, and $\theta=\pi/2$ otherwise. The $\KP_{\St}$ distance between the centers of $D_R$ and $D_r$ is given by
\begin{subequations}
\begin{alignat}{2}
&\frac{d}{r}
    &&\tif r=R \label{eq:KPdistanceA}\\
&\text{and when $r<R$}&& \nonumber\\
&\frac12\log\frac{R+d}{R-d}
    &&\tif d\le R\tan(\theta/2) \label{eq:KPdistanceB}\\
&\frac12\left[\cot\frac{\theta}{2}\log\left(\frac{R}{r}\cos\theta\right)+\log\frac{1+\tan(\theta/2)}{1-\tan(\theta/2)}\right] &&\tif d> R\tan(\theta/2). \label{eq:KPdistanceC}
\end{alignat}
\end{subequations}
\end{lemma}

Observe that whenever~\eqref{eq:KPdistanceC} applies we have $\tan(\theta/2)\in (0,1)$.

\begin{proof}
Throughout this proof we refer to the disks $D_r$ and $D_R$ as well as their centers; these are the disks from Definition~\ref{definition-stadium}.

If $r=R$, then $\St$ is contained in an infinite strip of width $2r$. By  Remark~\ref{rem:ExtremalDiskDomainMonotone} and in light of Example~\ref{ex:ExtremalDiskInStrip}, at every point $z$ along the segment connecting the centers the extremal disk is $D(z,r)$. Then $\mu_\St(z)=\lambda_{D(0,r)}(0)$, from~\eqref{hyp-density-in-disk} the density is $1/r$, and integrating this along a segment of length $d$ yields the result in~\eqref{eq:KPdistanceA}.

Next assume $d\le R-r$; then $\theta=\pi/2$. In this case we clearly have $d<R \tan(\theta/2)$, and furthermore $D_r\subset D_R$ so $\St=D_R$. Like Example~\ref{ex:ExtremalDiskInDisk}, for every point in $\St$ the $\KP_{\St}$~extremal disk will be $D_R$. Thus the KP distance between the centers is the KP length of a radial segment of length $d$, with the center of $D_R$ as one endpoint. This distance is equivalent to $h_{D(0,R)}(0,d)$. Evaluating this by integrating~\eqref{hyp-density-in-disk} gives the formula in~\eqref{eq:KPdistanceB}.

Now assume that $R-r<d\le R\tan(\theta/2)$. We will show that the segment connecting the centers is contained in the convex hull of $\partial D_R\cap\partial\St$ in the hyperbolic metric on $D_R$, and thus $D_R$ is the extremal disk along the whole segment. Let $\Sh=\{ |\Arg z|<\theta\}$ as in Definition~\ref{def:infinite-sector}. Then it is easily verified that $\partial D_R\cap\partial\Sh=\{Re^{\pm i\theta}\cot\theta\}$ and $\partial D_R\cap\partial\St$ has endpoints $\{Re^{\pm i\theta}\cot\theta\}$. The convex hull of $\partial D_R\cap\partial\St$ in the hyperbolic metric on $D_R$ is the portion of $D_R$ bounded by $\partial D_R\cap\partial\St$ and $D(0,R\cot\theta)\cap\{|\Arg z|<\theta\}$ (see Example~\ref{ex:HullOfCircularArc}). The distance from the center of $D_R$, located at $R\csc\theta$, to the boundary of $D(0,R\cot\theta)$ along the real axis is $R\csc\theta-R\cot\theta=R\tan(\theta/2)$. Then because $d\le R\tan(\theta/2)$, the segment is contained in the convex hull, $D_R$ is the $\KP_{\St}$~extremal disk along the segment, and the center of $D_R$ is one endpoint of the segment. The $\KP_{\St}$ length can again be calculated as $h_{D(0,R)}(0,d)$. This completes the result in~\eqref{eq:KPdistanceB}.

Finally, assume $d>R\tan(\theta/2)$, and thus the segment connecting the centers extends beyond the convex hull of $\partial D_R\cap\partial\St$. We will divide the segment into a \emph{proximal segment} $[r\csc\theta,R\cot\theta]$ and a \emph{distal segment} $[R\cot\theta,R\csc\theta]$, where proximal and distal indicate relative position with respect to the vertex at 0. The distal segment will have as its extremal disk $D_R$, and as before we calculate the length as

\begin{equation}
\KP_\St(R\cot\theta,R\csc\theta)=\int_0^{R(\csc\theta-\cot\theta)} \lambda_{D(0,R)}(z)\,dz =\frac12\log\frac{1+\tan(\theta/2)}{1-\tan(\theta/2)}.
\end{equation}

For the proximal segment we rely on work done by Herron, Ma, and Minda~\cite[p. 331]{HMM2005}. They produced a formula for the KP metric density at any point in an infinite sector. After adjusting the notation, the curvature convention, and taking advantage of the simplification that our segment is along the central axis, the formula is

\begin{equation}\label{HMMSectorFormula}
\mu_{\Sh}(z)=\frac{1}{2z}\cot(\theta/2). 
\end{equation}

We need to show that the $\KP_{\Sh}$~extremal disk for every point on the proximal segment is contained in $\St$. Then, by Remark~\ref{rem:ExtremalDiskDomainMonotone} it will also be the extremal disk in $\St$. This will justify using the infinite sector formula in~\eqref{HMMSectorFormula} to give the KP density in our stadium $\St$. It will suffice to show that the extremal disk for the two endpoints of the proximal segment are in $\St$.

The proximal endpoint of the proximal segment is $r\csc\theta$, constructing the extremal disk in $\Sh$ for $r\csc\theta$ gives a disk tangent to $\partial\Sh$ at $\{re^{\pm i\theta}\csc\theta\}$. The disk $D_r$ is tangent to $\partial\Sh$ at $\{re^{\pm i \theta}\cot\theta\}$, and because $\csc\theta<\cot\theta$ in $(0,\pi)$ the $\KP_{\Sh}$~extremal disk for the endpoint $r\csc\theta$ is far enough from the vertex to be contained in $\St$. The other endpoint is $R\cot\theta$, and we have already seen that $\forall z\in\St\colon |z|\ge R\cot\theta$, the $\KP_\St$~extremal disk is $D_R$. 

We now calculate the KP length of the proximal segment in $\St$ by integrating the density given in~\eqref{HMMSectorFormula}:
\begin{equation}
\int_{r\csc\theta}^{R\cot\theta} \frac{1}{2z}\cot(\theta/2)\,dz = \frac12 \cot(\theta/2)\log\left(\frac{R}{r}\cos\theta\right).
\end{equation}
Combining the proximal and distal lengths completes the proof,
\begin{equation}
\KP_\St(r\csc\theta,R\csc\theta)= \frac12 \left[ \cot(\theta/2)\log\left(\frac{R}{r}\cos\theta\right)+\log\frac{1+\tan(\theta/2)}{1-\tan(\theta/2)} \right]. 
\qedhere
\end{equation} 
\end{proof}

\begin{lemma}\label{lem:DistanceMonotone}
Let $h(r_1,r_2,d)$ denote the hyperbolic distance between the centers in a stadium $\St(r_1,r_2,d)$. Then $h(r_1,r_2,d)$ is an increasing function in $d$.
\end{lemma}

\begin{proof}
Fix $r_1$ and $r_2$. Let $d_1<d_2$ and define $\lambda=d_2/d_1$. Dilating $\St(r_1,r_2,d_1)$ by a factor of $\lambda$ and observing conformal invariance of the hyperbolic metric we have 
\begin{equation}\label{eqn:DistanceMonotone-one}
h(r_1,r_2,d_1)=h(\lambda r_1,\lambda r_2, d_2).
\end{equation}
Now consider $\St(r_1,r_2,d_2)$ and $\St(\lambda r_1,\lambda r_2, d_2)$. After a rigid motion, the segments connecting the centers of two stadia are coincident and $\St(r_1,r_2,d_2)\subset\St(\lambda r_1,\lambda r_2, d_2)$. Then by the monotonicity of the hyperbolic metric~\eqref{eq:HyperbolicDomainMonotonic} we have
\begin{equation}\label{eqn:DistanceMonotone-two}
h(\lambda r_1,\lambda r_2,d_2)\le h(r_1,r_2,d_2),
\end{equation}
and combining~\eqref{eqn:DistanceMonotone-one} and~\eqref{eqn:DistanceMonotone-two} gives the result.
\end{proof}

\begin{lemma}\label{ArcsineDefined}
Let $\Omega$ satisfy the $(R_O,R_I,R_C)$ condition in Definition~\ref{RadiiDefinition}. Then \[|R_I-R_C|\le R_O-R_C.\]
\end{lemma}
\begin{proof}
If $R_C \le R_I$, then we are trying to show $R_I - R_C \le R_O - R_C$. It is clear from the definitions that $R_I\le R_O$, so the inequality is verified in this case. 

Now assume $R_I < R_C$, we want to show $R_C-R_I \le R_O-R_C$. There exists a point $\xi \in \partial\Omega$ such that $|\xi|=R_I$. The smoothness of $\partial\Omega$ implies that it must be tangent to $\partial D(0,R_I)$ at $\xi$.

Since $D(0,R_I)$ and $\partial\Omega$ are tangent at $\xi$, by Remark~\ref{rem:CurvatureRadius} there must be a disk $D(a,R_C)\subset\Omega$ with $\xi$ as a boundary point. By the definition of $R_C$ and the fact that $\Omega$ is convex, $\partial D(a,R_C)$ must be tangent to $\partial\Omega$ at $\xi$ and therefore tangent to the disk $D(0,R_I)$ at $\xi$. Observe that $D(a,R_C)\subset\Omega\subset D(0,R_O)$. It follows that $2R_C\le\diam\Omega\le R_O+R_I$, and thus $R_C-R_I \le R_O-R_C$.
\end{proof}

The necessary condition in Lemma~\ref{ArcsineDefined} turns out also to be sufficient for the existence of such $\Omega$.

\begin{lemma}\label{lemma-construction}
We can construct a domain $\Omega$ satisfying the $(R_O,R_I,R_C)$ condition for an arbitrary $R_O, R_I, R_C$ so long as they satisfy the relationship $|R_I-R_C|\le R_O-R_C$.
\end{lemma}

\begin{proof}
If $R_I>R_C$, let $K$ be the closed Euclidean convex hull of the set $D(0,R_I-R_C)\cup\{R_O-R_C\}$. Otherwise, let $K$ be the line segment $[R_C-R_I,R_O-R_C]$. Define $\Omega=\bigcup\limits_{z\in K} D(z,R_C)$.

By construction, $\Omega$ is convex and has a $C^{1,1}$-smooth boundary. More specifically, $\Omega$ is a stadium in the sense of Definition~\ref{definition-stadium}. That $\Omega$ has the required values of $R_O$ and $R_I$ is a consequence of the fact that $-R_I$ and $R_O$ are boundary points of $\Omega$, and that $D(0,R_I)\subset\Omega\subset D(0,R_O)$.
\end{proof}

Our main result of this section provides an upper bound on the hyperbolic distance from the base point 0 to any point $a$ such that $D(a,R_C)\subset\Omega$. This bound is given in terms of the outer, inner, and curvature radii of $\Omega$. 

\begin{theorem}\label{DistanceBounds}
Let $\Omega$ satisfy the $(R_O,R_I,R_C)$ condition in Definition~\ref{RadiiDefinition} and let $a$ be any point such that $D(a,R_C)\subset\Omega$. Let $R=\max(R_C, R_I)$, $r=\min(R_C,R_I)$, $d=R_O-R_C$, and $\theta=\arcsin\frac{R-r}{d}$. Define $F(R_O,R_I,R_C)$ as
\begin{subequations}
\begin{alignat}{2}
&\frac{d}{R}
    &&\tif r=R,\label{eq:DistanceBoundsTheoremA}\\
\text{when $r<R$} & && \nonumber\\
&\frac12\log \frac{R+d}{R-d}
    &&\tif d\le R\tan\frac{\theta}{2}\label{eq:DistanceBoundsTheoremB}\\
&\frac12\left[\cot\frac{\theta}{2}\log\left(\frac{R}{r}\cos\theta\right)+\log\frac{1+\tan(\theta/2)}{1-\tan(\theta/2)}\right]\quad    &&\tif d> R\tan\frac{\theta}{2}.
    \label{eq:DistanceBoundsTheoremC}
\end{alignat}
\end{subequations}
Then $h_\Omega(0, a)\le F(R_O,R_I,R_C)$. 
\end{theorem}

\begin{proof}
First, note that Lemma~\ref{ArcsineDefined} allows us to define $\theta$ in this way. Observe that $D(a,R_C)\subset\Omega\subset D(0,R_O)$, thus $|a|+R_C\le R_O$ and $|a|\le d$. Since $\Omega$ is a convex domain containing the disks $D(0, R_I)$ and $D(a,R_C)$, it contains the corresponding stadium $\St(R_I, R_C, |a|)$. For containment the position of the stadium is important, so to be clear $\St$ is the convex hull of $D(0,R_I)\cup D(a,R_C)$. It follows from the domain monotonicity of the hyperbolic metric and Lemma~\ref{lem:DistanceMonotone} that
\begin{equation}\label{A}
h_\Omega(0,a) \le h(R_I,R_C,|a|)
= h(R,r,|a|) \le h(R,r,d)
\end{equation}
where $h$ taking three arguments is as in Lemma~\ref{lem:DistanceMonotone}. Since the hyperbolic distance is majorized by the KP distance~\eqref{KPcomparison}, the claim follows from the explicit formulas for KP distance from Lemma~\ref{KPdistance}.

In each of the three cases we find the longest possible line segment from $0$ to an allowable $a$, construct $\St(R_I,R_C,R_O-R_C)=\St(R,r,d)\subset\Omega$ around the segment, and find its $\KP_{\St}$ length. The formula in~\eqref{eq:DistanceBoundsTheoremA} corresponds to the case where the KP metric has constant density along the segment. The formula in~\eqref{eq:DistanceBoundsTheoremB} corresponds to the case where the $\KP_{\St}$ extremal disk is the same at every point of the segment. The formula in~\eqref{eq:DistanceBoundsTheoremC} corresponds to the case where the extremal disk and density vary along the segment.
\end{proof}

\section{Estimates for the derivative of a conformal map}\label{sec:derivative} 

We are now ready to prove the main result. 

\begin{theorem}\label{Main}
Let $\Omega$ satisfy the $(R_O,R_I,R_C)$ condition in Definition~\ref{RadiiDefinition}. Then for any conformal map $f\colon  \D\onto\Omega$ fixing $0$ we have
\begin{equation}\label{Main-theorem-formula}
\|f'\|_{H^\infty} \le R_Ce^{2F(R_O,R_I,R_C)}
\end{equation}
where $F$ is as in Theorem~\ref{DistanceBounds}.
\end{theorem}

\begin{proof}
By assumption, $\Omega$ has a smooth boundary and $f'$ exists on $\partial\D$. Take any number $L > R_C e^{2F(R_O,R_I,R_C)}$. It suffices to show that $|f'|\le L$ in $\D$, which we will do by proving that this inequality holds on the boundary and then applying the maximum principle. More specifically, it suffices to show
\begin{equation}\label{p1}
    \limsup_{|z|\nearrow 1}\frac{\dist(f(z),\partial\Omega)}{1-|z|}\le L.
\end{equation}

Fix $z\in\D$. Let $d=\dist(f(z),\partial\Omega)$, since we are interested in the limit as $d\to 0$, we may assume $d<R_C$. We will show that 
\begin{equation}\label{p2}
    \frac{d}{L} \le 1-|z|
\end{equation}
for sufficiently small $d$, thus establishing the inequality in~\eqref{p1}.

We choose a point $w\in\partial\Omega$ such that $|f(z)-w|=d$. By Remark~\ref{rem:CurvatureRadius} there is a disk $D=D(a,R_C)$ that has $w$ on its boundary and is contained in $\Omega$. Observe that the smoothness of $\partial\Omega$ and $\partial D$ at $w$ require that $f(z)$ lie on the radius of $D$ that ends at $w$, and therefore $|f(z)-a|=R_C-d$. Keeping in mind that the hyperbolic metric is monotone with respect to domain and that the formula for hyperbolic distance in a disk along a radius is well known, observe
\begin{equation}\label{p3}
h_{\Omega}(f(z),a)\le h_{D(a,R_C)}(f(z),a)=\frac12\log\frac{2R_C-d}{d}<\frac12\log\frac{2R_C}{d}.
\end{equation}
Next we estimate $h_{\Omega}(a,0)$ using the KP estimate from Theorem~\ref{DistanceBounds}:
\begin{equation}\label{p4} h_{\Omega}(a,0) \le F(R_O,R_I,R_C). \end{equation}
Now suppose for the sake of contradiction that~\eqref{p2} is false, this implies
\[ 1-|z|<d/L \ssp{3}{\tand} 1+|z|>2-d/L.\]
By conformal invariance of the hyperbolic metric, $h_{\Omega}(f(z),0)=h_\D(z,0)$, so

\begin{equation}\label{p5}
    h_\Omega(f(z),0)=\frac12\log\frac{1+|z|}{1-|z|} > \frac12\log\frac{2-d/L}{d/L}.
\end{equation}

Using the triangle inequality to combine this with~\eqref{p3} and~\eqref{p4}, we get
\[ \frac12\log\frac{2-d/L}{d/L} < \frac12\log\frac{2R_C}{d}+F(R_O,R_I,R_C) \]
which can be be rearranged to
\[ L-\frac{d}{2}<R_C e^{2F(R_O,R_I,R_C)}. \]
But $R_C e^{2F(R_O,R_I,R_C)}<L$, so we have a contradiction when $d$ is sufficiently small. This contradiction proves~\eqref{p2}.
\end{proof}

\section{Examples}\label{sec:examples}

\begin{example}
Let $\Omega=\{|z|<r\}$, then $R_O=R_I=R_C=r$. Theorem~\ref{Main} says that for all conformal $f\colon \D\onto\Omega$ fixing 0, $\|f'\|_{H^\infty} \le re^0=r$. The function $f(z)=rz$ shows the bound is attained in this case.
\end{example}

This can be generalized to show that the bound in Theorem~\ref{Main} is sharp whenever $\Omega$ is a disk containing 0.

\begin{proposition}\label{prop-sharp in disks} Let $\Omega=D(a, r)$ with $0\le |a|<r$. Then the bound in Theorem~\ref{Main} is sharp for a conformal map $f\colon \D\onto\Omega$. 
\end{proposition}

\begin{proof} 
After a rotation about the origin, we may assume $0\le a <r$. Then $\Omega$ satisfies the condition ($R_O=r+a,R_I=r-a,R_C=r$), and one can check that Theorem~\ref{Main} gives 
\begin{equation}\label{prop-disk-bound}
\|f'\|_{H^\infty}\le r\frac{r+a}{r-a}.    
\end{equation}

Take \[f_1(z)=\frac{z-\frac{a}{r}}{1-\frac{a}{r}z}, \quad f_2(z)=z+\frac{a}{r}, \quad\text{and }f_3(z)=rz.\] Then the conformal mapping $f_3\circ f_2\circ f_1\colon\D\onto D(a,r)$ fixes 0, and its derivative attains the bound in~\eqref{prop-disk-bound} at $z=1$.
\end{proof}

To illustrate the improvement over the bound given in~\eqref{old-bound}, we close with two more examples. 

\begin{example}
One of the examples considered in~\cite{Kovalev2017} is a rounded triangle with $R_O=0.6,R_I=0.5,R_C=0.4$. 
For this domain, the bound~\eqref{old-bound} is 
$\|f'\|_{H^\infty}\le 0.977$. Theorem~\ref{Main} improves this to $\|f'\|_{H^\infty}\le 0.931$.
\end{example}

\begin{example} 
For the domain in Proposition~\ref{prop-sharp in disks} with $0\le a<r$ the bound~\eqref{old-bound} is 
\begin{equation}\label{old-in-disks}
\|f'\| \le \frac{r^3}{(r-a)^2}.     
\end{equation}
The ratio of two bounds~\eqref{prop-disk-bound} and ~\eqref{old-in-disks} tends to $0$ as $a\to r$, indicating a substantial improvement. 
For a specific example, let $a=1$ and $r=2$, so $\Omega=D(1,2)$. The bound in~\eqref{prop-disk-bound} becomes $\|f'\|_{H^\infty} \le 6$, which is sharp as noted above. In contrast~\eqref{old-bound} gives $\|f'\|_{H^\infty} \le 8$ for this example. 
\end{example}

\section*{Acknowledgments}

This paper is based on a part of a PhD thesis written by the author under the
supervision of Leonid Kovalev. The author thanks the anonymous referees for many useful suggestions in revising this paper.

\bibliography{references.bib} 
\bibliographystyle{plain} 
\end{document}